\documentclass[a4paper,11pt,english]{smfart}

\usepackage{amssymb}
\usepackage{url}
\usepackage[T1]{fontenc}
\RequirePackage{calrsfs}
\DeclareSymbolFont{rsfscript}{OMS}{rsfs}{m}{b}
\DeclareSymbolFontAlphabet{\mathrsfs}{rsfscript}
\usepackage[utopia,expert]{mathdesign}
\usepackage{lscape}

\usepackage{enumitem}

\usepackage{todonotes}

\usepackage[vcentermath]{youngtab}

\usepackage[leftbars]{changebar}

\usepackage{palatino}
\usepackage{rotating}
\usepackage{graphicx}

\usepackage{tikz-cd}

\usepackage{enumerate}

\usepackage{amscd}

\usepackage{color}
\definecolor{shadecolor}{gray}{0.90}

\def\proofname{D\'emonstration}

\input xypic
\xyoption{all}
\xyoption{arc}

\makeindex

\def\indexname{Index des notations}

\newtheorem{theo}{Theorem}[section]
\def\thetheo{\thesection.\arabic{theo}}
\newtheorem{prop}[theo]{Proposition}

\newtheorem{lem}[theo]{Lemma}

\newtheorem{coro}[theo]{Corollary}

\newtheorem{defi}[theo]{Definition}

\renewcommand\theequation{\thetheo}
\def\equat{\refstepcounter{theo}\begin{equation}}
\def\endequat{\end{equation}}

\renewcommand\thetable{\thetheo}

\renewcommand\thesection{\arabic{section}}
\renewcommand\thesubsection{\thesection.\Alph{subsection}}

\def\contentsname{Table des mati\`eres}
\def\listfigurename{Liste des figures}
\def\listtablename{Liste des tables}
\def\appendixname{Appendice}

\def\refname{R\'ef\'erences}

    \def\CM{{\mathbb{C}}}

  \def\gG{{\mathfrak g}}  
  \def\hG{{\mathfrak h}}

\def\LG{{\mathfrak L}}

    \def\ZM{{\mathbb{Z}}}

    \def\CC{{\mathcal{C}}}
  \def\db{{\mathbf d}}

\def\Gb{{\mathbf G}}    
\def\Hb{{\mathbf H}}

\def\Lb{{\mathbf L}}

    \def\OC{{\mathcal{O}}}
\def\Pb{{\mathbf P}}    \def\PC{{\mathcal{P}}}

    \def\XC{{\mathcal{X}}}
    \def\YC{{\mathcal{Y}}}
\def\Zb{{\mathbf Z}}

\def\Irm{{\mathrm{I}}}

\def\Trm{{\mathrm{T}}}

    \def\ZCB{{\boldsymbol{\mathcal{Z}}}}

\def\Ati{{\tilde{A}}}

\def\Qov{{\overline{Q}}}

\def\a{\alpha}

\def\d{\delta}

\def\e{\varepsilon}

\def\L{\Lambda}

\def\O{\Omega}

\def\th{\theta}

\def\mub{{\boldsymbol{\mu}}}

\DeclareMathOperator{\diag}{{\mathrm{diag}}}

\DeclareMathOperator{\End}{{\mathrm{End}}}

\DeclareMathOperator{\Hom}{{\mathrm{Hom}}}

\DeclareMathOperator{\Id}{{\mathrm{Id}}}

\DeclareMathOperator{\Ker}{{\mathrm{Ker}}}

\DeclareMathOperator{\Res}{{\mathrm{Res}}}

\DeclareMathOperator{\Core}{{\mathrm{Core}}}

\def\to{\rightarrow}
\def\longto{\longrightarrow}

\def\fonction#1#2#3#4#5{\begin{array}{rccc}
{#1} : & {#2} & \longto & {#3}  \\
& {#4} & \longmapsto & {#5} 
\end{array}}

\def\vide{\varnothing}
\def\DS{\displaystyle}

\def\finl{~$\blacksquare$}

\def\lexp#1#2{\kern\scriptspace\vphantom{#2}^{#1}\kern-\scriptspace#2}
\def\le{\hspace{0.1em}\mathop{\leqslant}\nolimits\hspace{0.1em}}
\def\ge{\hspace{0.1em}\mathop{\geqslant}\nolimits\hspace{0.1em}}

\mathchardef\inferieur="321E
\mathchardef\superieur="321F

\def\eqna{\begin{eqnarray*}}
\def\endeqna{\end{eqnarray*}}

\def\aff{\mathrm{aff}}

\catcode`\@=11
\long\def\@car#1#2\@nil{#1}
\long\def\@first#1#2{#1}
\long\def\@second#1#2{#2}
\long\def\ifempty#1{\expandafter\ifx\@car#1@\@nil @\@empty
  \expandafter\@first\else\expandafter\@second\fi}
\catcode`\@=12

\def\GL{{\mathrm{GL}}}

\DeclareMathOperator{\Reff}{Ref}

\theoremstyle{remark}
\newtheorem{rema}[theo]{Remark}

\newtheorem{exemple}[theo]{Example}

\theoremstyle{plain}

\def\BIL{LR}
\def\GAUCHE{L}
\def\CAR{CAR}
\def\FAM{FAM}

\def\reg{{\mathrm{reg}}}

\def\xyinj{\ar@{^{(}->}}
\def\xysur{\ar@{->>}}

\bigskip

\DeclareMathOperator{\Rep}{{\mathrm{Rep}}}

\makeatletter
\def\hlinewd#1{%
\noalign{\ifnum0=`}\fi\hrule \@height #1 %
\futurelet\reserved@a\@xhline}
\makeatother

\newlength\epaisLigne

\usepackage{multirow}

\newcommand{\longiso}{\stackrel{\sim}{\longrightarrow}}

\def\rdim{\operatorname{\dim^{\rm reg}}}

\makeatletter
\def\hlinewd#1{%
\noalign{\ifnum0=`}\fi\hrule \@height #1 %
\futurelet\reserved@a\@xhline}
\makeatother

\usepackage{multirow}

\usepackage{lscape}

\usepackage{pstricks}

\def\GL{\operatorname{\Gb\Lb}\nolimits}

\def\hQ{{\widehat Q}}

\begin{document}

\title{Symplectic leaves of \\Calogero-Moser spaces of type $G(\ell,1,n)$}



\author{{\sc Ruslan Maksimau}}

\address{
	Institut Montpelli\'erain Alexander Grothendieck (CNRS: UMR 5149), 
	Universit\'e de Montpellier,
	Case Courrier 051,
	Place Eug\`ene Bataillon,
	34095 MONTPELLIER Cedex,
	FRANCE} 


\date{\today}

\maketitle

\pagestyle{myheadings}
\markboth{\sc R. Maksimau}{\sc Symplectic leaves of Calogero-Moser spaces of type $G(\ell,1,n)$}

\bigskip

\begin{abstract} 
We study symplectic leaves of Calogero-Moser spaces of type $G(\ell,1,n)$. We prove that the normalization of the closure of each symplectic leaf is isomorphic to some Calogero-Moser space. We also give a nice combinatorial parameterization of the symplectic leaves.
\end{abstract}

\section{Introduction}

This preprint is a part of an unfinished paper. This is a natural continuation of \cite{BM}.
We study symplectic leaves of Calogero-Moser spaces of type $G(\ell,1,n)$ under the assumption that the parameter $a$ is nonzero. 

One of the main results of the paper is Theorem \ref{thm:norm-cl-sympl-leaf}. There we prove that the normalization of the closure of each symplectic leaf is isomorphic to some Calogero-Moser space, which confirms a conjecture given in \cite{Bon}. We also give in \S \ref{subs:param-sympl} a nice combinatorial parameterization of the symplectic leaves.

Gwyn Bellamy and Travis Schedler informed me that they also proved Theorem \ref{thm:norm-cl-sympl-leaf} independently. It is expected that this preprint will become a part of a joint paper with Gwyn Bellamy and Travis Schedler.

\bigskip

\section{Combinatorics}

	\bigskip
	
	\subsection{Partitions}
	Assume $\ell\in \ZM_{>0}\cup \{\infty\}$ and $n\in\ZM_{\geqslant 0}$. 
	A \emph{partition} is a tuple $\lambda=(\lambda_1,\lambda_2,\dots,\lambda_r)$ of positive integers 
	(with no fixed length) such that $\lambda_1\geqslant\lambda_2\geqslant\dots\geqslant\lambda_r$, $r\geqslant 0$. 
	Set $|\lambda|=\sum_{i=1}^r \lambda_i$. If $|\lambda|=n$, we say that $\lambda$ is a partition of $n$. 
	
	Denote by $\PC$ (resp. $\PC[n]$) be the set of all partitions (resp. the set of all partitions of $n$). 
	By convention, $\PC[0]$ contains one (empty) partition (it has $r=0$). We will identify partitions with Young diagrams.
	The partition $\lambda$ corresponds to a Young diagram with $r$ lines such that the $i$th line contains $\lambda_i$ boxes. 
	For example the partition $(4,2,1)$ corresponds to the Young diagram
	
	$$
	\yng(4,2,1)
	$$
	
	Let us use the following convention: for $\ell=\infty$ we have $\ZM/\ell\ZM=\ZM$.
	We say that a box $b$ of the Young diagram is \emph{at position $(r,s)$} if it is in the line $r$ and column $s$. 
	The $\ell$-\emph{residue} of the box $b$ is the 
	number $s-r$ modulo $\ell$. (We say that the integer $s-r$ is the $\infty$-residue of the box $b$). Then we obtain a map 
	$$
	\Res_\ell\colon \PC\to \ZM^{\ZM/\ell\ZM},\qquad \lambda\mapsto \Res_\ell(\lambda),
	$$ 
	such that for each $i\in \ZM/\ell\ZM$ the number of boxes with $\ell$-residue $i$ in $\lambda$ is $(\Res_\ell(\lambda))_i$. 
	(In particular, we obtain a map $\Res_\infty\colon \PC\to \ZM^{\ZM}$.) For $\ell=\infty$, we mean that $\ZM^{\ZM/\ell\ZM}=\ZM^\ZM$ is the direct sum (and not the direct product) of $\ZM$ copies of $\ZM$. In other words, our convention is that for an element $\db=(d_i)_{i\in \ZM}\in\ZM^{\ZM}$, only a finite number of integers $d_i$ is nonzero.
	
	\bigskip
	
	\begin{exemple}
		For the partition $\lambda=(4,2,1)$ and $\ell=3$ the $3$-residues of the boxes are
		$$
		\young(0120,20,1)
		$$
		In this case we have $\Res_\ell(\lambda)=(3,2,2)$ because there are three boxes with residue $0$, two boxes with residue 
		$1$ and two boxes with residue $2$.\finl
	\end{exemple}
	
	\bigskip
	
	We say that a box of a Young diagram is \emph{removable} if it has no boxes on the right and on the bottom. In other words, a box $b$ is removable for $\lambda$ if $\lambda\backslash b$ is still a Young diagram. We say that a box $b$ is \emph{addable} for $\lambda$ if $b$ is not a box of $\lambda$ and $\lambda\cup b$ is still a Young diagram.  For $i\in \ZM/\ell\ZM$, we say that a box is $i$-addable or respectively $i$-removable if it is an addable or respectiely removable box with $\ell$-residue $i$.

	For $\lambda,\mu\in\PC$, we write $\mu\leqslant \lambda$ if the Young diagram of $\mu$ can be obtained from 
	the Young diagram of $\lambda$ by removing a sequence of removable boxes.

	\subsection{$\ell$-cores}
	Assume $\ell\in \ZM_{>0}$.
	\begin{defi}
		We say that the partition $\lambda$ is an $\ell$-core if there is no partition $\mu\leqslant\lambda$ such that 
		the Young diagram of $\mu$ differs from the Young diagram of $\lambda$ by $\ell$ boxes with $\ell$ different $\ell$-residues.
	\end{defi}
	
	\bigskip
	
	See \cite{BJV} for more details about the combinatorics of $\ell$-cores. Let $\CC_\ell\subset \PC$ be the set of $\ell$-cores. 
	Set $\CC_\ell[n]=\PC[n]\cap \CC_\ell$.
	
	If a partition $\lambda$ is not an $\ell$-core, then we can get a smaller Young diagram from its Young diagram by 
	removing $\ell$ boxes with different $\ell$-residues. We can repeat this operation again and again until we get an $\ell$-core. 
	It is well-known, that the $\ell$-core that we get is independent of the choice of the boxes that we remove. Then we get an application
	$$
	\Core_\ell\colon \PC\to \CC_\ell.
	$$
	If $\mu=\Core_\ell(\lambda)$, we will say that the partition $\mu$ is {\it the $\ell$-core} of the partition $\lambda$.
	
	\bigskip
	
	\begin{exemple}
		The partition $(4,2,1)$ from the previous example is not a $3$-core because it is possible to remove three bottom boxed. We get 
		$$
		\young(0120)
		$$
		But this is still not a $3$-core because we can remove three more boxes and we get 
		$$
		\young(0)
		$$
		This shows that the partition $(1)$ is the $3$-core of the partition $(4,2,1)$.\finl
	\end{exemple}
	
		\medskip
	Let $\delta_\ell$ denote the constant family $\delta_\ell=(1)_{i \in \ZM/\ell\ZM} \in \ZM^{\ZM/\ell\ZM}$.
	\begin{rema}
		\label{rem:same-lcores}
	Assume that we have $\mu=\Core_\ell(\lambda)$ and $\mu$ is obtained from $\lambda$ by removing $r\ell$ boxes. Then we have $\Res_\ell(\lambda)=\Res_\ell(\mu)+r\delta_\ell$. In partiluar, if we have two partitions $\lambda_1$ and $\lambda_2$ with the same $\ell$-cores and such that $|\lambda_1|=|\lambda_2|$, then they have the same $\ell$-residues. More generally, if two partition $\lambda_1$ and $\lambda_2$ have the same $\ell$-cores then we have $\Res_\ell(\lambda_1)=\Res_\ell(\lambda_2)+r\delta_\ell$, where $r=(|\lambda_1|-|\lambda_2|)/\ell$. 
	
	\end{rema}

	For $\nu\in \CC_\ell$, set $\PC_\nu=\{\lambda\in \PC;~\Core_\ell(\lambda)=\nu\}$ and $\PC_\nu[n]=\PC_\nu\cap\PC[n]$.
	
	\bigskip

	\subsection{Action of the affine Weyl group}\label{sub:action-weyl}
	Assume $\ell\in \ZM_{>0}$. Let $W_\ell^\aff$ denote the affine Weyl group of type $\Ati_{\ell-1}$. For $\ell\geqslant 2$ it is the 
	Coxeter group with associated Coxeter system $(W_\ell^\aff,S_\ell^\aff)$, where 
	$S_\ell^\aff=\{s_i~|~i \in \ZM/\ell\ZM\}$ and the Coxeter graph whose vertices are elements of $\ZM/\ell\ZM$ and we have an edge between $i$ and $i+1$ for each $i\in \ZM/\ell\ZM$. We also extend this notion to the case $\ell=1$ by setting $W_1^\aff=1$.  We denote by $l$ the length function $l\colon W^\aff_\ell\to \ZM_{\geqslant 0}$.
	
	
	
	The non-affine Weyl group $W_\ell$ (isomorphic to the symmetric group $\mathfrak S_\ell$) is a parabolic subgroup of $W_\ell^\aff$ generated by $s_1,\ldots, s_{\ell-1}$ (for $\ell=1$ we mean that $W_1=1$).
	
	\medskip


	Consider the Lie algebra $\gG_\ell=\mathfrak{sl}_\ell(\CM)$ and its affine version  $\widehat{\gG}_\ell=\widehat{\mathfrak{sl}}_\ell(\CM)=\mathfrak{sl}_\ell(\CM)[t,t^{-1}]\oplus \CM \mathbf{1}\oplus \CM\partial$.
	Let $\hG\subset \gG$ be the Cartan subalgebra formed by the diagonal 
	matrices and set $\widehat{\hG}=\hG\oplus \CM \mathbf{1}\oplus \CM\partial$. 
	
	The $\CM$-vector space $\widehat{\hG}^*$ has a basis $(\alpha_0,~ \alpha_1,\ldots,~\alpha_{\ell-1},\Lambda_0)$, where $\alpha_0,~ \alpha_1,\ldots,~\alpha_{\ell-1}$ are the simple roots of $\widehat{\gG}_\ell$ and $\Lambda_0$ is such that $\Lambda_0$ annihilates $\hG$ and $\partial$ and $\Lambda_0(\mathbf{1})=1$. Denote by $R^\aff_\ell$ and $R_\ell$ the affine and the non-affine root lattices respectively (i.e., $R_\ell^\aff$ is the $\ZM$-lattice generated by $\alpha_0,~ \alpha_1,\ldots,~\alpha_{\ell-1}$ and $R_\ell$ is the sublattice generated by $\alpha_1,\ldots,~\alpha_{\ell-1}$.)

	Following~\cite{lusztig}, 
	we define two actions of $W_\ell^\aff$: a non-linear one on $\ZM^{\ZM/\ell\ZM}$, 
	and a linear one on $\CM^{\ZM/\ell\ZM}$. If $\ell=1$, there is nothing to define so we may assume that 
	$\ell \ge 2$. If $\db=(d_i)_{i \in \ZM/\ell\ZM} \in \ZM^{\ZM/\ell\ZM}$ and if 
	$j \in \ZM/\ell\ZM$, we set $s_j(\db)=(d_i')_{i \in \ZM/\ell\ZM}$, where 
	$$d_i'=
	\begin{cases}
	d_i & \text{if $i \neq j$,}\\
	\d_{j0} + d_{i+1}+d_{i-1}-d_i & \text{if $i=j$.}\\
	\end{cases}
	$$
	
	\bigskip
	\begin{rema}
		\label{rem:twisted_action}
		We can identify $\ZM^{\ZM/\ell\ZM}$ with the root lattice $R_\ell^\aff$ by 
		$
		\db\mapsto \sum_{i\in \ZM/\ell\ZM}d_i\alpha_i.
		$
		Under this identification the element $\delta_\ell\in\ZM^{\ZM/\ell\ZM}$ corresponds to the imaginary root of $R_\ell^\aff$ that we also denote by $\delta_\ell$.
		
		Beware, the action considered here is not the usual action of $W_\ell^\aff$ on the root lattice. 
		When we have $w(\db)=\db'$ with respect to the action define above, this corresponds to $w(\Lambda_0-\db)=\Lambda_0-\db'$ for 
		the usual action of $W_\ell^\aff$ on $\widehat \hG^*$.\finl
	\end{rema}

	If $\th=(\th_i)_{i \in \ZM/\ell\ZM} \in \CM^{\ZM/\ell\ZM}$, we set $s_j(\th)=(\th_i')_{i \in \ZM/\ell\ZM}$, where
	$$\th_i'=
	\begin{cases}
	\th_i & \text{if $i \not\in \{j-1,j,j+1\}$,}\\
	\th_j+\th_i & \text{if $i \in \{j-1,j+1\}$,}\\
	-\th_i & \text{if $i=j$.}
	\end{cases}
	$$
	It is readily seen that these definitions on generators extend to an action 
	of the whole group $W_\ell^\aff$. 
	We also define a pairing $\ZM^{\ZM/\ell\ZM} \times \CM^{\ZM/\ell\ZM} \to \CM$, $(\db,\th) \mapsto \db \cdot \th$, 
	where 
	$$\db \cdot \th = \sum_{i \in \ZM/\ell\ZM} d_i \th_i.$$
	Then
	$$
	s_j(\db) \cdot s_j(\th) = (\db \cdot \th) - \d_{j0} \th_0.
	$$

	\bigskip
	

	\begin{rema}
	\label{rem:W-act-lcores}
	\cite[Sec.~3]{BJV} defined an $W^\aff_\ell$-action on $\CC_\ell$. Let us recall this construction. Fix $i\in \ZM/\ell\ZM$ and $\nu\in\CC_\ell$.
	\begin{itemize}
		\item[(1)] Assume that $\nu$ has neither $i$-removable boxes nor $i$-addable boxes, then we have $s_i(\nu)=\nu$.
		\item[(2)] Assume that $\nu$ has no $i$-removable boxes and has at least one $i$-addable box. Then $s_i(\nu)$ is obtained from $\nu$ by addition of all $i$-addable boxes.
		\item[(3)] Assume that $\nu$ has no $i$-addable boxes and has at least one $i$-removable box. Then $s_i(\nu)$ is obtained from $\nu$ by removing of all $i$-removable boxes.
		\item[(4)] The situation when the $\ell$-core $\nu$ has an $i$-addable box and an $i$-removable box at the same time is impossible.
	\end{itemize}
    \end{rema}

	By construction, the map $\Res_\ell\colon \CC_\ell\to \ZM^{\ZM/\ell\ZM}$ is $W^\aff_\ell$-invariant. Moreover, the $\ell$-residue of the empty partition is zero. 
	The stabilizer of the empty partition in $W^\aff_\ell$ is $W_\ell$ and the stabilizer of $0\in\ZM^{\ZM/\ell\ZM}$ in  $W^\aff_\ell$ is also $W_\ell$. This implies that we have $W^\aff_\ell$-invariant bijections
	$$
	\begin{array}{ccccl}
	W^\aff_\ell/W_\ell&\simeq& \CC_\ell&\simeq& W^\aff_\ell\cdot 0\subset \ZM^{\ZM/\ell\ZM}\\
	wW_\ell&\mapsto& w(\emptyset)& \mapsto & w(0)
	\end{array}
	$$
Since $\Res_\ell$ is a $W^\aff_\ell$-invariant map and $\Res_\ell(\emptyset)=0$, then the bijection $\CC_\ell\simeq W^\aff_\ell\cdot 0$ is given by the map $\Res_\ell$. In particular, we see that an element $\ZM^{\ZM/\ell\ZM}$ is a residue of an $\ell$-core if and only if it is in the $W^\aff_\ell$-orbit of $0$.

Moreover, since we have $w(\db+n\delta_\ell)=w(\db)+n\delta_\ell$ and since each $W^\aff_\ell$-orbit in $\ZM^{\ZM/\ell\ZM}$ contains exactly one element of the form $n\delta_\ell$ (see \cite[Lem.~2.8]{BM}), each element $\db\in\ZM^{\ZM/\ell\ZM}$ has a unique presentation in the form 
\begin{equation}
\label{eq:d-via-core-n}
\db=\Res_\ell(\nu)+n\delta_\ell,\qquad\nu\in\CC_\ell, n\in\ZM.
\end{equation}

The following lemma is a reformulation of \cite[Remark~3.2.3]{BJV}.
\begin{lem}
	\label{lem:nu-w-i}	
	Fix $\nu\in\CC_\ell$ and $i\in \ZM/\ell\ZM$. Let $w$ be the unique element of $W^\aff_\ell$ such that $w(\emptyset)=\nu$ and such that $w$ is the shortest element in the coset $wW_\ell\in W^\aff_\ell/W_\ell$. 
	The the situations $(1)$, $(2)$, $(3)$ in Remark \ref{rem:W-act-lcores} are equivalent to the following situations $(1)$, $(2)$, $(3)$ respectively:
	\begin{itemize}
		\item[$(1)$] $s_iw\in wW_\ell$ and $l(s_iw)>l(w)$,
		\item[$(2)$] $s_iw\not\in wW_\ell$ and $l(s_iw)>l(w)$,
		\item[$(3)$] $s_iw\not\in wW_\ell$ and $l(s_iw)<l(w)$.
	\end{itemize}
	
\end{lem}

	\bigskip
	
	For $\db=(d_i)_{i\in \ZM/\ell\ZM}\in \ZM^{\ZM/\ell\ZM}$ we set $|\db|=\sum_{i\in \ZM/\ell\ZM}d_i$.

\subsection{Another presentation of the affine Weyl group}
	Recall that the affine Weyl group has another presentation. We have $W_\ell^\aff=W_\ell\ltimes R_\ell$.
	For each $\alpha\in R_\ell$, denote by $t_\alpha$ the image of $\alpha$ in $W^\aff_\ell$. 
	Each element of $W_\ell^\aff$ can be written in a unique way in the form $w\cdot t_\alpha$, where $w\in W_\ell$ and $\alpha\in R_\ell$. 
	We can also extend the notation $t_\alpha$ to $\alpha\in R_\ell^\aff$ by setting $t_\alpha:=t_{\pi(\alpha)}$ for each $\alpha\in R_\ell^\aff$, where $\pi$ is the following map
	$$
	\pi\colon R_\ell^\aff\to R_\ell^\aff/\ZM\delta_\ell\simeq R_\ell.
	$$
	In the following lemma we identify $\ZM^{\ZM/\ell\ZM}$ with $R_\ell^\aff$.
	
	\bigskip
	
	\begin{lem}
		\label{lem:affWeyl_translation-roots}
		Assume $\alpha\in R_\ell$ and $d\in \ZM^{\ZM/\ell\ZM}$. Then we have $t_\alpha(d) \equiv d-\alpha \mod \ZM\delta_\ell$.
	\end{lem}
	
	\bigskip
	\begin{proof}
		This statement is a partial case of \cite[(6.5.2)]{Kac} (see also Remark \ref{rem:twisted_action}).
	\end{proof}
	
	\bigskip

	\bigskip

	Consider the $\ZM$-linear map
	$$
	R^\aff_\ell\to \CM^{\ZM/\ell\ZM}, \quad \db\mapsto \overline \db,
	$$
	given by
	$$
	(\overline{\alpha_r})_i=2\delta_{i,r}-\delta_{i,r+1}-\delta_{i,r-1}.
	$$
	The kernel of this map is $\ZM\delta_\ell$. Set 
	$$
	\Sigma(\th)=\sum_{i \in \ZM/\ell\ZM} \th_i.
	$$
	
	\bigskip
	\begin{lem}
		\label{lem:affWeyl_translation-theta}
		For each $\alpha\in R_\ell$ and $\theta\in \CM^{\ZM/\ell\ZM}$, we have $t_\alpha(\theta)=\theta+\Sigma(\th)\overline\alpha$.
	\end{lem}
	\bigskip
	\begin{proof}
		The $W^\aff_\ell$-action on $\CM^{\ZM/\ell\ZM}$ defined above coincides with the (usual) action of $W^\aff_\ell$ on the dual of the span of $\alpha_0,\alpha_1,\ldots,\alpha_{\ell-1}$ in $\widehat\hG^*$. The statement follows from \cite[(6.5.2)]{Kac}. 
	\end{proof}

\subsection{$J$-cores} Fix a subset $J\subset \ZM/\ell\ZM$. 

\begin{defi}
	We say that a box of a Young tableau is \emph{$J$-removable} if it is removable and its residue is in $J$. We say that a Young tableau is a \emph{$J$-core} if it has no $J$-removable boxes. Denote by $\CC_J$ the set of all $J$-cores. 
	
	To each partition $\lambda\in \PC$ we can associate a partition $\Core_J(\lambda)\in \CC_J$ obtained from it by removing $J$-removable boxes (probably in several steps). The result $\Core_J(\lambda)$ does not depend on the order of operations.
\end{defi}


\medskip
\begin{lem}
\label{lem:lcore-of-Jcore}
	For each $\mu\in \CC_J$, we have $\Core_\ell(\mu)\in \CC_J$.
\end{lem}
\begin{proof}
This statement is quite obvious when we see the partition $\mu$ as an abacus, see for example \cite[\S~ 2]{BJV} for then definition of an abacus.

However we can give another proof based on the representation theory of quivers and the results of \S 3. Fix some $J$-standard $\th\in \CM^{\ZM/\ell\ZM}$. Then, since $\nu$ is a $J$-core, the representaion $A_\mu$ constructed in \S \ref{subs:C-fixed-points} is simple by Lemma \ref{lem:structure-Amu}. Then the dimension vector $\Res_\ell(\mu)$ of this representation is in $E_\th$.

Now, let $\nu$ be the $\ell$-core of $\mu$. Assume that $\nu$ is obtained from $\mu$ by removing $r\ell$ boxes. The we have $\Res_\ell(\mu)=\Res_\ell(\nu)+r\delta_\ell\in E_\th$. Now, Lemma \ref{lem:descr-Eth} implies that $\nu$ is a $J$-core.

\end{proof}

\section{Preliminaries on quiver varieties}\label{sec:quiver}

By an algebraic variety, we mean a reduced scheme of finite type over $\CM$. 

\medskip

\subsection{Quiver varieties}\label{sub:quiver}

Assume $\ell\in \ZM_{>0}\cup\{\infty\}$.

Let $Q_\ell$ denote the cyclic quiver with $\ell$ vertices, defined as follows:
\begin{itemize}
\item[$\bullet$] Vertices: $i \in \ZM/\ell\ZM$ (recall that we use the convention that for $\ell=\infty$ we have $\ZM/\ell\ZM=\ZM$). 

\item[$\bullet$] Arrows: $y_i : i \longto i+1$, $i \in \ZM/\ell\ZM$. 
\end{itemize}
We denote by $\Qov_\ell$ the double quiver of $Q_\ell$ that is, the quiver obtained from $Q_\ell$ by adding an arrow 
$x_i : i+1 \to i$ for all $i \in \ZM/\ell\ZM$.




\medskip

Now, let $\db=(d_i)_{i \in \ZM/\ell\ZM}$ be a family of elements of $\ZM_{\geqslant 0}$. (For $\ell=\infty$ we always assume additionally that $\db$ has a finite number of nonzero components.)

Let $\Rep(\Qov_\ell,\db)$ be the variety of representations 
of $\Qov_\ell$ in the family of vector spaces $(\CM^{d_i})_{i \in \ZM/\ell\ZM}$. More precisely, we have $\Rep(\Qov_\ell,\db)=\bigoplus_{i\in \ZM/\ell\ZM}\Hom(\CM^{d_i},\CM^{d_{i+1}})\oplus \Hom(\CM^{d_{i+1}},\CM^{d_{i}})$. An element of $\Rep(\Qov_\ell,\db)$ is a couple $(X,Y)$ where 
$$
X=(X_i)_{i\in \ZM/\ell\ZM},\quad X_i\in \Hom(\CM^{d_{i+1}}, \CM^{d_{i}}) \qquad Y=(Y_i)_{i\in \ZM/\ell\ZM},\quad Y_i\in \Hom(\CM^{d_i}, \CM^{d_{i+1}}).
$$
We denote by $\Gb\Lb(\db)$ the direct product
$$\Gb\Lb(\db)=\prod_{i \in \ZM/\ell\ZM} \Gb\Lb_{d_i}(\CM),$$
The group $\Gb\Lb(\db)$ acts on $\Rep(\Qov_\ell,\db)$. The orbits 
are the isomorphism classes of representations of $\Qov_\ell$ of dimension vector $\db$. 
We denote by 
$$\fonction{\mu_\db}{\Rep(\Qov_\ell,\db)}{\bigoplus_{i \in \ZM/\ell\ZM}\End(\CM^{d_i})}
{(X_i,Y_i)_{i \in \ZM/\ell\ZM}}{(X_iY_i - Y_{i-1}X_{i-1})_{i \in \ZM/\ell\ZM}}
$$
the corresponding {\it moment map}. Finally, if $\th=(\th_i)_{i \in \ZM/\ell\ZM}$ is a family of complex numbers, 
we denote by $\Irm_\th(\db)$ the family $(\th_i \Id_{\CM^{d_i}})_{i \in \ZM/\ell\ZM}$.
Finally, we set
$$\YC^0_\th=\mu_\db^{-1}(\Irm_\th(\db))\qquad\text{and}\qquad
\XC^0_\th(\db)=\YC^0_\th(\db)/\!\!/\Gb\Lb(\db).$$ 
Note that the variety $\XC^0_\th(\db)$ is not empty only in the case $\db\cdot \th=0$.
Note that $\YC^0_\th(\db)$ is endowed with a $\CM^\times$-action: if $\xi \in \CM^\times$, we set 
$$\xi \cdot (X,Y)=(\xi^{-1}X,\xi Y).$$
This action commutes with the action of $\Gb\Lb(\db)$ and the moment map is constant on 
$\CM^\times$-orbits, so it induces a $\CM^\times$-action on $\XC^0_\th(\db)$. 

\bigskip

Now, we give a framed version $\XC_\th(\db)$ of the variety $\XC^0_\th(\db)$. Let $\hQ_\ell$ be the quiver obtained from $\Qov_\ell$ by adding a new vertex $\infty$ and arrows $0\to \infty$ and $\infty\to 0$.

For each dimension vector $\db$ for the quiver $\Qov$ we consider the dimension vector $\widehat\db$ such that $\widehat\db$ has dimension $1$ at the vertex $\infty$ and the same dimension as $\db$ for other vertices.

Let $\Rep(\hQ_\ell,\widehat\db)$ be the variety of representations 
of $\hQ_\ell$ with dimension vector $\widehat\db$. More precisely, we have 
$$
\Rep(\hQ_\ell,\widehat\db)=\Hom(\CM^{d_0},\CM)\oplus\Hom(\CM,\CM^{d_0})\oplus\bigoplus_{i\in \ZM/\ell\ZM}\Hom(\CM^{d_i},\CM^{d_{i+1}})\oplus \Hom(\CM^{d_{i+1}},\CM^{d_{i}}).
$$ 
An element of $\Rep(\hQ_\ell,\widehat\db)$ is of the form $(X,Y,x,y)$ where 
$$
X=(X_i)_{i\in \ZM/\ell\ZM},\quad X_i\in \Hom(\CM^{d_{i+1}}, \CM^{d_{i}}) \qquad Y=(Y_i)_{i\in \ZM/\ell\ZM},\quad Y_i\in \Hom(\CM^{d_i}, \CM^{d_{i+1}}),
$$
$$ 
x\in \Hom(\CM,\CM^{d_0}),\qquad y\in \Hom(\CM^{d_0},\CM).
$$

The group $\Gb\Lb(\db)$ acts on $\Rep(\hQ_\ell,\widehat\db)$. The orbits 
are the isomorphism classes of representations of $\hQ_\ell$ of dimension vector $\widehat\db$. 
We denote by 
$$\fonction{\widehat\mu_\db}{\Rep(\hQ_\ell,\widehat\db)}{\bigoplus_{i \in \ZM/\ell\ZM} 
\End(\CM^{d_i})}{(X_i,Y_i,x,y)_{i \in \ZM/\ell\ZM}}{(X_iY_i - Y_{i-1}X_{i-1}+\delta_{i,0}xy)_{i \in \ZM/\ell\ZM}}$$
the corresponding {\it moment map}. 
Finally, we set
$$\YC_\th(\db)=\widehat\mu_\db^{-1}(\Irm_\th(\db))\qquad\text{and}\qquad
\XC^0_\th(\db)=\YC^0_\th(\db)/\!\!/\Gb\Lb(\db).$$ 
Note that in the case $\db\cdot \th=0$ we have an obvious isomorphism $\XC_\th(\db)=\XC^0_\th(\db)$. 
Note that $\YC_\th(\db)$ is endowed with a $\CM^\times$-action: if $\xi \in \CM^\times$, we set 
$$\xi \cdot (X,Y,x,y)=(\xi^{-1}X,\xi Y,x,y).$$
This action commutes with the action of $\Gb\Lb(\db)$ and the moment map is constant on 
$\CM^\times$-orbits, so it induces a $\CM^\times$-action on $\XC^0_\th(\db)$. 

\begin{rema}\label{rem:dim-negative}
We extend the definition of $\XC_\th(\db)$ to the case where $\db \in \ZM^{\ZM/\ell\ZM}$ 
by the convention that $\XC_\th(\db)=\vide$ whenever at least one of the $d_i$'s is negative.\finl
\end{rema}

Let $\Rep(\widehat Q_\ell)$ be the category of representations of the quiver $\widehat Q_\ell$. We can see each element of $\Rep(\widehat Q_\ell,\db)$ as an object in $\Rep(\widehat Q_\ell)$ with dimension vector $\widehat \db$.
Now, assume $\ell\in \ZM_{\geqslant 0}$. 

\begin{defi}
\label{def:map-i-inf}
Consider the following map $\iota\colon \Rep(\widehat Q_\infty)\to  \Rep(\widehat Q_\ell)$.

For each finite dimensional representation $(X,Y,x,y)$ of $\widehat Q_\infty$ in the vector space $V=\bigoplus_{j\in \ZM} V_j$ we can associate a representataion $(X',Y',x',y')$ of $\widehat Q_\ell$ in the vector space $V'=\bigoplus_{i\in \ZM/\ell\ZM} V'_i$ where
$$
V'_i = \bigoplus_{\substack{j \in \ZM \\ j \equiv i \mod\ell}} V_j,\qquad X'_i = \bigoplus_{\substack{j \in \ZM \\ j \equiv i \mod\ell}} X_j,\quad  Y'_i = \bigoplus_{\substack{j \in \ZM \\ j \equiv i \mod\ell}} Y_j,
$$
$x'$ is the composition of $x$ with the natural map $V_0\to V'_0$, $y'$ is the composition of $y$ with the natural map $V'_0\to V_0$.
\end{defi}

\subsection{Lusztig's isomorphism}
\label{subs:Lusztig-iso}
We use the $W_\ell^{\rm aff}$-actions on $\ZM^{\ZM/\ell\ZM}$ and $\CM^{\ZM/\ell\ZM}$ defined in Section \ref{sub:action-weyl}.

It is proved in~\cite[Corollary~3.6]{lusztig} that
\equat\label{eq:iso-lusztig}
\XC_{s_j(\th)}(s_j(\db)) \simeq \XC_\th(\db)\qquad \mbox{ if }\theta_j\ne 0.
\endequat
Note that this isomorphism takes into account the convention of Remark~\ref{rem:dim-negative}.

The isomorphism above motivates to consider the following equivalence relation on the set $\ZM^{\ZM/\ell\ZM} \times \CM^{\ZM/\ell\ZM}$. Let $\sim$ be the transitive closure of 
$$
(\db,\theta)\sim (s_i(\db),s_i(\theta)), \qquad \theta_i\ne 0.
$$
The isomorphism \eqref{eq:iso-lusztig} implies that if $(\db,\theta)\sim(\db',\theta')$, then we have an isomorphism of algebraic varieties $\XC_\theta(\db)\simeq \XC_{\theta'}(\db')$.

\begin{rema}
\label{rem:equiv=shortest}
Let $W_\th$ be the stibilizer of $\th$ in $W^\aff_\ell$. Assume that $\th$ is such that $W_\th$ is a parabolic subgroup of $W^\aff_\ell$. 
Then we can describe the set of couples that are equivalent to $(\db,\th)$ in the following way. They are of the form $(w(\db),w(\th))$ where $w$ is the element of $W^\aff_\ell$ such that $w$ is the shortest element in the class $wW_\th\in W^\aff_\ell/W_\th$.
\end{rema}

\bigskip
\subsection{Calogero-Moser space} 
\label{sub:CM}
We fix a $\CM$-vector space $V$ 
		of finite dimension $n$ and a finite subgroup $W$ of $\GL_\CM(V)$. We set
		$$\Reff(W)=\{s \in W~|~\dim_\CM V^s=n-1\}$$
		and we assume that
		$W=\langle \Reff(W) \rangle$.

\medskip

We set $\e : W \to \CM^\times$, $w \mapsto \det(w)$. 
If $s \in \Reff(W)$, we denote by $\a_s^\vee$ and $\a_s$ two elements of $V$ and $V^*$ respectively 
such that $V^s=\Ker(\a_s)$ and $V^{* s}=\Ker(\a_s^\vee)$, where $\a_s^\vee$ is viewed as a linear 
form on $V^*$.

Let us fix a function $c : \Reff(W) \to \CM$ which 
is invariant under conjugacy. We define the $\CM$-algebra $\Hb_c$ to be the quotient 
of the algebra $\Trm(V\oplus V^*)\rtimes W$ (the semi-direct product of the tensor algebra 
$\Trm(V \oplus V^*)$ with the group $W$) 
by the relations 
$$\begin{cases}
[x,x']=[y,y']=0,\\
[x,y]=\DS{\sum_{s\in\Reff(W)}(\e(s)-1)c_s
	\frac{\langle y,\alpha_s\rangle\langle\alpha_s^\vee,x\rangle}{\langle\alpha_s^\vee,\alpha_s\rangle}s},
\end{cases}
$$
for all $x$ ,$x'\in V^*$, $y$, $y'\in V$. The algebra $\Hb_c$ is called the 
{\it rational Cherednik algebra at $t=0$}. 

The first commutation relations imply that 
we have morphisms of algebras $\CM[V] \to \Hb_c$ and $\CM[V^*] \to \Hb_c$.

We denote by $\Zb_c$ the center of $\Hb_c$: it is well-known~\cite[Lemma~3.5]{EG} that 
$\Zb_c$ is an integral domain, which is integrally closed and contains 
$\CM[V]^W$ and $\CM[V^*]^W$ as subalgebras (so it contains $\Pb=\CM[V]^W \otimes \CM[V^*]^W$), 
and which is a free $\Pb$-module of rank $|W|$. We denote by $\ZCB_{\! c}$ the 
algebraic variety whose ring of regular functions $\CM[\ZCB_{\! c}]$ is $\Zb_c$: 
this is the {\it Calogero-Moser space} associated with the datum $(V,W,c)$. 
If necessary, we will write $\ZCB_{\! c}(V,W)$ for $\ZCB_{\! c}$.

\subsection{Quiver varieties vs Calogero-Moser spaces}\label{sub:quiver-cm}

Assume that $n \ge 2$, that $V=\CM^n$ and that $W=G(\ell,1,n)$. Recall that $G(\ell,1,n)$ is the group of monomial 
matrices with coefficients in $\mub_\ell$ (the group of $\ell$-th root of unity in $\CM^\times$).

We fix a primitive $\ell$-th root of unity $\zeta$. We denote by $s$ the permutation matrix corresponding to the 
transposition $(1,2)$ and we set
$$t=\diag(\zeta,1,\dots,1) \in W.$$
Then $s$, $t$, $t^2$,\dots, $t^{\ell-1}$ is a set of representatives of conjugacy 
classes of reflections of $W$. We set for simplification
$$a=c_s\qquad\text{and}\qquad k_j=\frac{1}{\ell} \sum_{i=1}^{\ell-1} \zeta^{-i(j-1)} c_{t^i}$$
for $j \in \ZM/\ell\ZM$. Then
\equat\label{eq:k}
k_0+\cdots + k_{\ell-1} = 0\qquad \text{and}\qquad c_{t^i}=\sum_{j \in \ZM/\ell\ZM} \zeta^{i(j-1)} k_j
\endequat
for $1 \le i \le \ell-1$. Finally, if $i \in \ZM/\ell\ZM$, we set

\equat
\th_i=
\begin{cases}
k_{-i}-k_{1-i} & \text{if $i \neq 0$,}\\
-a +k_0-k_1 & \text{if $i=0$.}\\
\end{cases}
\endequat\label{eq:theta-k}
and $\th=(\th_i)_{i \in \ZM/\ell\ZM}$. 

The following result is proved in~\cite[Theorem~3.10]{gordon quiver}. (Note that our $k_i$ is related with 
Gordon's $H_i$ via $H_i=k_{-i}-k_{1-i}$.)

\bigskip

\begin{prop}\label{prop:quiver-cm}
There is a $\CM^\times$-equivariant 
isomorphism of varieties
$$\ZCB_{\! c} \longiso \XC_\th(n\delta_\ell).$$
\end{prop}

In the isomorphism above, the parameter $a$ of the variety $\ZCB_{\! c}$ corresponds to $~-(\sum_{i\in \ZM/\ell\ZM}\th_i)$ for $\XC_\th(n\delta_\ell)$. So, we will sometimes use the notation $a=-\Sigma(\theta)=-(\sum_{i\in \ZM/\ell\ZM}\th_i)$ when we speak about an arbitrary quiver variety $\XC_\th(\db)$. Note also that $a$ is invariant under the transformation of the parameter $\th\mapsto s_j(\th)$. In this paper, we will often assume $a\ne 0$.

\begin{rema}
	All statements in \S \ref{sub:quiver-cm} make also sense for $n=1$ with the following modifications. We have no transposition $s$, so we have no parameter $a$. On the other hand, for $n=1$, the variety $\XC_\th(n\delta_\ell)$ does not depend on $\theta_0$. Proposition \ref{prop:quiver-cm} is true for an arbitrary choice of $a$ in \eqref{eq:theta-k}.
	
	We can also use the convention that for $n=0$ the Calogero-Moser space is a point. Then Proposition \ref{prop:quiver-cm} still holds.
\end{rema}

Recall also from~\cite[\S{11}]{EG} the following result, which follows from Proposition \ref{prop:quiver-cm}.

\bigskip

\begin{lem}\label{lem:dim-CM}
	If $n \ge 0$, then $\XC_\th(n\delta_\ell)$ is normal and of dimension $2n$.
\end{lem}

\subsection{Simple representations in $\Rep_{\th}(\Qov_\ell)$}

From now on we assume $a\ne 0$.

Denote by $\Rep_{\theta}(\Qov_\ell)$ the additive category of representations $(X,Y)$ of $\Qov_\ell$ satisfying the moment map relations $\mu_\db(X,Y)=\Irm_\th(\db)$, where $\db$ is the dimension vector of the representation $(X,Y)$. In this section we give an explicit description of the set $\Sigma_\theta$ of dimension vectors of simple representations in $\Rep_{\theta}(\Qov_\ell)$. This description is done in much more generality in \cite{CB}. In this section, we precise how this description looks like in our particular case: the cyclic quiver and $a\ne 0$.

By \cite[Theorem 5.8]{CB} there are two types of indecomposable representations in $\Rep_\th(\Qov_\ell)$:
\begin{itemize}
\item[\textbullet] representations whose dimension vectors are positive roots,
\item[\textbullet] representations whose dimension vectors are of the form $r\delta_\ell$ for $r>0$.
\end{itemize}

Since we assume $a\ne 0$, the second situation is not possible. Now, let us give a precise description of the dimension vectors of simple representations. 

Let $R^+\subset\ZM^{\ZM/\ell\ZM}$ be the set of positive real roots. Set $R^+_\th=\{\db\in R^+;\db\cdot \theta=0\}$.
The following proposition is the special case of \cite[Theorem~1.2]{CB}.
\begin{prop}
The dimension vectors of simple representations in $\Rep_{\th}(\Qov_\ell)$ are exactly the elements of $R^+_\th$ that are not presented as sums of (two or more) elements of $R^+_\th$.
\end{prop}

\begin{coro}
\label{coro:unique-ss}
For each dimesnion vector $\db\in \ZM_{\geqslant 0}^{\ZM/\ell\ZM}$, there exists at most one (up to isomorphism) semisimple representation in $\Rep_{\th}(\Qov_\ell)$ with dimension vector $\db$.
\end{coro}
\begin{proof}
The statement is equivalent to the fact that the variety $\XC^0_\th(\db)$ contains at most one point. 

First, assume $\db\in \Sigma_\th$.  Then, since $a\ne 0$, $\db$ is a positive root. This implies that there is exactly one (up to isomorphism) simple representation in $\Rep_{\th}(\Qov_\ell)$ with dimension vector $\db$ (see the introduction in \cite{CB}).

Now, consider an arbitraty $\db\in\ZM_{\geqslant 0}^{\ZM/\ell\ZM}$. Then there is a finite number of possibilities to decompose $\db$ in a sum of elements of $\Sigma_\th$. This implies that $\XC^0_\th(\db)$ has a finite number of points. The variety $\XC^0_\th(\db)$ is irreducible if it is non-empty by \cite[Cor.~1.4]{CB2}. So, the variety $\XC^0_\th(\db)$ contains at most one point. 

\end{proof}

\begin{coro}
The elements of $\Sigma_\th$ are $\ZM$-linearly independent.
\end{coro}
\begin{proof}
Since $\XC^0_\th(\db)$ contains at most one point, there is at most one way (up to permutation) to decompose $\db$ in a sum of elements of $\Sigma_\th$.
\end{proof}

Denote by $\Sigma\Sigma_\th$ the set of sums of element of $\Sigma_\th$ (we also allow an empty sum, so we assume $0\in \Sigma\Sigma_\th$). In other words, the set $\Sigma\Sigma_\th$ is the set of all dimension vectors $\db$ such that there exists a representation in $\Rep_\th(\Qov_\ell)$ of dimension vector $\db$. For each $\db\in \Sigma\Sigma_\th$, denote by $L(\db)$ the unique semisimple representataion in $\Rep_\th(\Qov_\ell)$.

\subsection{Symplectic leaves}
\label{subs:symp-leaves}

Denote by $\Rep_{\theta}(\hQ_\ell)$ the category of representations $(X,Y,x,y)$ of $\hQ$ whose dimension vector is of the form $\widehat\db$ for some $\db\in \ZM^{\ZM/\ell\ZM}$ and satisfying the moment map relations $\widehat\mu_\db(X,Y)=\Irm_\th(\db)$. This category is not additive because we have imposed that the representations have dimension $1$ at the vertex $\infty$. However, it does make sence to add an object of $\Rep_{\theta}(\hQ_\ell)$ and an object of $\Rep_{\theta}(\Qov_\ell)$ getting an object of $\Rep_{\theta}(\hQ_\ell)$.

An object $M$ of $\Rep_{\theta}(\hQ_\ell)$ is indecomposable as a representation of the quiver $\hQ_\ell$ if and only if the only possible decomposition $M=M_0\oplus M_1$ with $M_0\in \Rep_{\theta}(\hQ_\ell)$ and $M_1\in \Rep_{\theta}(\Qov_\ell)$ is $M=M\oplus 0$.

Denote by $E_\th$ the set of all possible dimension vectors $\db\in \ZM^{\ZM/\ell\ZM}$ such that there exists a simple representation in $\Rep_{\theta}(\hQ_\ell)$ with dimension vector $\widehat\db$. Sometimes we can write $E_{\th,\ell}$ instead of $E_\th$ to emphasize $\ell$.

\begin{rema}
\label{rk:couple-equiv-nd}
Assume $\db\in E_\th$. Then, by Lemma \cite[Lemma~7.2]{CB}, the couple $(\db,\th)$ is equivalent to a couple of the form $(n\delta_\ell,\th')$ with $n\geqslant 0$. In particuler, by  Propostion \ref{prop:quiver-cm}, the variety $\XC_\th(\db)$ is isomorphic to the Calogero-Moser space.
\end{rema}

Each object $M\in\Rep_{\theta}(\Qov_\ell)$ has a unique decomposition $M=M_0\oplus M_1$ such that $M_0\in \Rep_{\theta}(\hQ_\ell)$, $M_1\in \Rep_{\theta}(\Qov_\ell)$ and $M_0$ is indecomposable. Set $\dim^{\rm reg}M=\dim M_0\in \ZM^{\ZM/\ell\ZM}$.

Take a point $[M]\in\XC_\th(\db)$ presented by a semisimple representation $M\in \Rep_{\theta}(\hQ_\ell)$.

\begin{lem}
\label{lem:symple-leaves-descr}
Two points of $[M], [M']\in \XC_\th(\db)$ are in the same symplectic leaf if and only if we have $\rdim(M)=\rdim(M')$.
\end{lem}
\begin{proof}
Let us decompse $M$ in a direct sum of simple representations $M=\bigoplus_{r=0}^k M_r$, where $M_0\in \Rep_{\theta}(\hQ_\ell)$ and other summands are in $\Rep_{\theta}(\Qov_\ell)$.

Once we know the dimension vector $\db'$ of $M_0$, we know automatically $k$ and the dimension vectors of $M_1, M_2, \ldots,M_k$ (up to a permutation) because by Corollary \ref{coro:unique-ss}, there is a unique semisimple representation in $\Rep_{\theta}(\Qov_\ell)$ of dimension vector $\db-\db'$. Then the statement follows from the description of symplectic leaves given in \cite[Theorem~1.9]{bellamy2}.
\end{proof}

For two dimension vectors $\db$ and $\db'$ we set $\LG^\db_{\db'}=\{[M]\in \XC_\th(\db);~\rdim(M)=\db'\}$. By Lemma \ref{lem:symple-leaves-descr} $\LG^\db_{\db'}$ is either a symplectic leaf of $\XC_\th(\db)$ or is empty.

\begin{lem}
\label{lem:order-sym-leaves}
The symplectic leaves $\LG^\db_{\db'}\subset \XC_\th(\db)$ define a finite stratification of $\XC_\th(\db)$ into locally closed subsets.
For two simplectic leaves $\LG^\db_{\db'}$ and $\LG^\db_{\db''}$ of $\XC_\th(\db)$ we have $\LG^\db_{\db'}\subset \overline{\LG^\db_{\db''}}$ if and only if $\db''-\db'\in \Sigma\Sigma_\th$.
\end{lem}
\begin{proof}
This statement is a special case of \cite[Prop.~3.6]{bellamy2}.

Let us give some details. 
Let $M', M''\in \Rep_\th(\hQ_\ell)$ be simple representations with dimension vectors $\widehat{\db'}$ and $\widehat{\db''}$ respectively. Then we have $[L(\db-\db')\oplus M']\in \LG^\db_{\db'}$ and $[L(\db-\db'')\oplus M'']\in \LG^\db_{\db''}$.

Assume that we have $\db''-\db'\in \Sigma\Sigma_\th$. Then we have $L(\db-\db')\simeq L(\db-\db'')\oplus L(\db''-\db')$. Then the stabilizer of the representation  $L(\db-\db'')\oplus M''$ in $\Gb\Lb(\db)$ is clearly contained in the stabilizer of the representation $L(\db-\db'')\oplus L(\db''-\db')\oplus M'$ in $\Gb\Lb(\db)$. Then by \cite[Prop.~3.6]{bellamy2}, we have $\LG^\db_{\db'}\subset \overline{\LG^\db_{\db''}}$.

Inversly, assume $\LG^\db_{\db'}\subset \overline{\LG^\db_{\db''}}$. Then, by \cite[Prop.~3.6]{bellamy2} there exists a semisimple representation $K\in \Rep_\th(\hQ)$ such that $[K]\in \LG^\db_{\db'}$ and the stabilizer of $K$ in $\Gb\Lb(\db)$ contains the stabilizer of $L(\db-\db'')\oplus M''$ in $\Gb\Lb(\db)$. Let $g$ be the element of the stabilizer of $L(\db-\db'')\oplus M''$ that acts on $M''$ by multiplication by $1$ and on $L(\db-\db'')$ by multiplication by $2$. Let $K_1$ and $K_2$ be the eigenspaces of $K$ with respect to the eigenvalues $1$ and $2$. Then, since $g$ is in the stabilizer of $K$, we get a decomposition $K=K_1\oplus K_2$ in a direct sum of subrepresentations. Moreover, we have $\dim K_1=\dim M''=\widehat{\db''}$. The representation $K_1$ can be decomposed as $K_1=K_{10}\oplus K_{11}$, where $K_{10}\in \Rep_\th(\hQ_\ell)$ is simple and $K_{11}\in \Rep_\th(\Qov_\ell)$. We clearly have $\dim K_{10}=\widehat{\db'}$. Then we get $\dim K_{11}=\widehat{\db''}-\widehat{\db'}=\db''-\db'$. This implies $\db''-\db'\in \Sigma\Sigma_\th$.
\end{proof}

\begin{prop}
\label{prop:can-decomp-of-d}
 \mbox{ }

 $a)$ For each dimension vector $\db$ such that $\XC_\th(\db)\ne \emptyset$, there is a decomposition $\db=\db_0+\db_1$ such that $\db_0\in E_\th$ and $\db_1\in\Sigma\Sigma_\th$ such that for any other decomposition $\db=\db'_0+\db'_1$ with $\db'_0\in E_\th$ and $\db'_1\in\Sigma\Sigma_\th$ we have $\db_0-\db'_0\in \Sigma\Sigma_\th$. 

$b)$ $\LG_{\db_0}$ is the unique open symplectic leaf in $\XC_\th(\db)$.

$c)$ We have an isomorphism of varieties 
$$
\XC_\th(\db_0)\simeq \XC_\th(\db),\qquad [M]\mapsto [M\oplus L(\db_1)].
$$
\end{prop}

\begin{proof}
	
By \cite[Cor.~1.45]{martino}, the smooth locus of $\XC_\th(\db)$ is a symplectic leaf. Then it should be of the form $\LG^\db_{\db_0}$ for some $\db_0$.
Since $\XC_\th(\db)$ is irreducible by \cite[Cor.~1.4]{CB2}, we have $\overline{\LG^\db_{\db_0}}=\XC_\th(\db)$.
Then, by Lemma \ref{lem:order-sym-leaves}  for any other symplectic leaf $\LG^\db_{\db'_0}$ we have $\db_0-\db'_0\in\Sigma\Sigma_\th$. This proves $a)$ and $b)$.

Part $c)$ follows from \cite[Theorem~1.1]{CB2}.
\end{proof}

Now, we set $\XC_\th(\db)^{\rm reg}=\LG^\db_{\db_0}$. Assume that $\db$ and $\db'$ are such that $\LG^\db_{\db'}$ is non-empty.

\begin{lem}
The normalization of the closure of $\LG^\db_{\db'}$ is isomorphic to $\XC_{\th}(\db')$. The normalization map is bijective.
\end{lem}
\begin{proof}

Consider the following homomorphism of algebraic varieties:
$$
\phi\colon\XC_{\th}(\db')\to \overline{\LG^\db_{\db'}},\qquad [M]\mapsto [M\oplus L(\db-\db')].
$$
Let us show that $\phi$ is bijective.


Fix a point $[N]\in \overline{\LG^\db_{\db'}}$ presented by a semisimple representation $N$.
We can decompose $N$ as $N=M\oplus L(\db-\db')$ for some semisimple $M\in \Rep_\th(\hQ_\ell)$. 
Then it is clear that the fibre $\phi^{-1}([N])$ contains a unique point: $[M]$. 

Moreover, the map $\phi$ restricts to an isomorphism $\XC_{\th}(\db')^{\reg}\to {\LG^\db_{\db'}}$, so $\phi$ is birational. Now, since $\XC_{\th}(\db')$ is normal, the map $\phi$ is a normalization. 

\end{proof}

\begin{coro}
\label{coro:norm-cl-leaf}
The normalization of the closure of each symplectic leaf $\LG^\db_{\db'}$ of the variety $\XC_{\th}(\db)$ is isomorphic to a variety of the form $\XC_{\th'}(r\delta_\ell)$ for some $r\geqslant 0$ and some $\th'\in \CM^{\ZM/\ell\ZM}$. 
\end{coro}
\begin{proof}

First of all, note that we have $\db'\in E_\th$. By Remark \ref{rk:couple-equiv-nd}, the pair $(\db',\th)$ is equivalent to some pair of the form $(r\delta_\ell,\th')$ where $r\geqslant 0$ and $\theta'\in \CM^{\ZM/\ell\ZM}$. Then the isomorphism \eqref{eq:iso-lusztig} yields $\XC_{\th}(\db')\simeq \XC_{\th'}(r\delta_\ell)$.
\end{proof}

Combining the corollary above with Proposition \ref{prop:quiver-cm} yields the following theorem.
\begin{theo}
	\label{thm:norm-cl-sympl-leaf}
The normalization of the closure of each symplectic leaf of the Calogero-Moser space of type $G(\ell,1,n)$ with $a\ne 0$ is isomorphic to a Calogero-Moser space of type $G(\ell,1,r)$ for some $r\in [0;n]$.
\end{theo}

\begin{rema}
Let us give an explicit relation between the parameters of the two Calogero-Moser spaces in the theorem above.

The original Calogero-Moser space is isomorphic to the quiver variety of the form $\XC_\theta(n\delta_\ell)$. Now, we consider the symplectic leaf $\LG^{n\delta_\ell}_{\db'}$, the normalization of its closure is isomorphic to  $\XC_\theta(\db')$.
Then, by Remark \ref{rk:couple-equiv-nd}, we can find $w\in W^\aff_\ell$ that realizes an equivalence between $(\db',\th)$ and $(w(\db'),w(\th))$ and such that $w(\db')$ is of the form $r\delta_\ell$. Set $\th'=w(\th)$. We have an isomorphism $\XC_\theta(\db')\simeq \XC_{\theta'}(r\delta_\ell)$. Since we have $w(\db')=r\delta_\ell$, then, by Lemma \ref{lem:affWeyl_translation-roots}, the element $w$ should be of the form $w=xt_{\db'}$, where $x\in W_\ell$.

Then, by Lemma \ref{lem:affWeyl_translation-theta}, we get $\th'=xt_{\db'}(\th)=x(\th-a\overline{\db'})$. Moreover, the action the element $x\in W_\ell$ on $\CM^{\ZM/\ell\ZM}$ corresponds to some permurtation of the parameters $k_0,k_1,\ldots,k_{\ell-1}$ (see \cite[Rem.~3.5]{BM}) and a permutation of the parameters does not change the Calogero-Moser space up to isomorphism, see \cite[Cor.~3.6]{BM}.

Now we see that the parameters $a, k_0,k_1,\ldots,k_{\ell-1}$ (corresponding to $\th$) of the original the Calogero-Moser space $\XC_\th(n\delta_\ell)$  are related with the parameters $a', k'_0,k'_1,\ldots,k'_{\ell-1}$ (corresponding to $\th'$) of the new Calogero-Moser space $\XC_{\th'}(r\delta_\ell)$ are related in the following way (up to a permutation of the parameters $k'_i$):
$$
a'=a,\qquad k'_i=k_i+(d'_{1-i}-d'_{-i}).
$$
In the case when $n$ or $r$ is equal to $1$, we can forget the parameter $a$ or $a'$ respectively. In the case $r=0$, the variety $\XC_{\th'}(r\delta)$ is just a point.

\end{rema}

\subsection{$\CM^\times$-fixed points}
\label{subs:C-fixed-points}

For each $J\subset \ZM/\ell\ZM$ we denote by $W_J$ the parabolic subgroup of $W^{\rm aff}_\ell$ generated by $s_i$ for $i\in J$. Let us say that $\th$ is \emph{$J$-standard} if the stibilizer $W_\th$ of $\th$ in $W^{\rm aff}_\ell$ is equal to $W_J$. We say that $\th\in\CM^{\ZM/\ell\ZM}$ is \emph{standard} it is $J$-standard for some $J\subset \ZM/\ell\ZM$. For a standard $\th$, the set $J$ is the set of indices $i\in \ZM/\ell\ZM$ such that $\th_i=0$.

Now, let us describe the $\CM^\times$-fixed points of $\XC_\th(\db)$. First of all, each couple $(\db,\th)$ is equivalent to a couple whose $\th$ is standard.

The following lemma is obvious.
	\begin{lem}
			\label{lem:Sigma-th-simple}
		Assume that $\th$ is $J$-standard. Then we have $\Sigma_\th=\{\alpha_i;~i\in J\}$.
	\end{lem}
	Let us now assume that $\theta$ is $J$-standard. 
For each partiction $\mu$, we construct a $\CM^\times$-fixed point in $\XC_{\th}(\Res_\ell(\mu))$. This construction is essentially the same as \cite[Section~5]{Przez}, however \cite{Przez} assumes that the variety $\XC_{\th}(\Res_\ell(\mu))$ is smooth and we don't need this assumption. 

Each partition $\mu$ can be described by some $k\in\ZM_{\geqslant 0}$ and $a_1,\ldots, a_k, b_1,\ldots,b_k\in \ZM_{\geqslant 0}$ where $k$ is maximal such that the Young diagram of $\mu$ contains a box in position $(k,k)$ and for each $r\in [1;k]$ there are $a_r$ boxes on the right of $(r,r)$ and $b_r$ boxes below $(r,r)$.
In other words, we see the Young diagram of the partition $\mu$ as a union of $k$ hooks. The box at position $(i,j)$ is in the $r$th hook if $\min(i,j)=r$. The numbers $a_r$ and $b_r$ are the lengths of the arm and of the leg of $r$th hook respectively.

For $i\in \ZM$, we use the convention that $\th_i$ means $\th_{(i\mod\ell)}$. Set $\beta_r=\sum_{i=-b_r}^{a_r}\theta_i$.

Let $V$ be a complex vector space with basis $\{v_{r,j};~r\in[1;k];~j\in[-b_r,a_r]\}$. It has a $\ZM$-grading $V=\bigoplus_{j\in \ZM} V_j$ such that $v_{r,j}\in V_j$. Consider two endomorphisms $X$ and $Y$ of this vector space given by
$$
X(v_{r,j})=
\begin{cases}
v_{r,j-1} &\mbox{ if } j>-b_r,\\
0 &\mbox{ if } j=-b_r,
\end{cases}
$$
and
$$
Y(v_{r,j})=
\begin{cases}
(\sum_{i=-b_r}^j \th_{i})v_{r,j+1}+\sum_{t>r} \beta_tv_{t,j+1} &\mbox{ if } j\in[-b_r,-1]\\
-(\sum_{i=j+1}^{a_r}\th_{i})v_{r,j+1}-\sum_{t<r}\beta_tv_{t,j-1} &\mbox{ if } j\in [0;a_r-1],\\
-\sum_{t<r}\beta_tv_{t,j-1} &\mbox{ if } j=a_r,\\
\end{cases}
$$

Consider also the linear maps $x\colon \CM\to V_0$ and $y\colon V_0\to \CM$ given by 
$$
x(1)=-\sum_{r=1}^k \beta_r v_{r,0}\qquad  \mbox{ and } \qquad y(v_{r,0})=1.
$$

Then $(X,Y,x,y)$ yields a representation $A^\infty_\mu$ of the quiver $\hQ_\infty$. Appliying the map $\iota$ as in Definition \ref{def:map-i-inf}, we get a representation $A_\mu$ of the quiver $\hQ_\ell$. It satisfies the moment map relation $\widehat\mu_\db(A_\mu)=I_\th(\db)$.

\begin{lem}
\label{lem:structure-Amu}

Assume that $\th$ is $J$-standard.

$(a)$ If $\mu$ is a $J$-core, then $A_\mu$ is simple.

$(b)$ Assume that  $b$ is a removable box of $\mu$ with $\ell$-residue $i\in J$. Then we either have a short exact sequence
$$
0\to  L(\alpha_i)\to A_{\mu}\to A_{\mu\backslash b}\to 0
$$
or we have a short exact sequence
$$
0\to A_{\mu\backslash b}\to A_{\mu}\to  L(\alpha_i)\to 0.
$$

\end{lem}
\begin{proof}
First, we prove $b)$. Assume that $b$ is the box as in the statement.  Assume that it is in the $r$th hook. Let $j$ be the $\infty$-residue of $b$.

Assume first $j<0$. We have $X(v_{r,j})=Y(v_{r,j})=0$.  Then the vector $v_{r,j}$ spans a subrepresentation isomorphic to $L(\alpha_i)$. We get a short exact sequence
$$
0\to  L(\alpha_j)\to A_{\mu}\to A_{\mu\backslash b}\to 0.
$$

Now, assume $j\geqslant 0$. Then we see that $A_{\mu\backslash b}$ is a subrepresentation of $A_{\mu}$. It is spanned by all basis vectors except $v_{r,j}$.
Then we have a short exact sequence
$$
0\to A_{\mu\backslash b}\to A_{\mu}\to  L(\alpha_j)\to 0.
$$

\medskip
Now, let us prove $a)$. First of all, we note that the assumption that $\theta$ is $J$-standard implies that if for some $a,b\in \ZM$, $a\leqslant b$ we have $\theta_a+\theta_{a+1}+\ldots+\theta_{b-1}+\theta_b=0$, then we have $\theta_a=\theta_{a+1}=\ldots=\theta_{b-1}=\theta_b=0$. 
If $\mu$ is a $J$-core, then the numbers $\beta_1, \beta_2,\ldots,\beta_k$ are nonzero. Indeed, if some $\beta_r$ is zero, then $\beta_k$ is also zero. Then the $\ell$-residues of all boxes of the $k$h hook are in $J$. In particuler, the $k$th hook contatins a removable box whose residue is in $J$. This contradicts to the fact that $\mu$ is a $J$-core.

In view of Lemma \ref{lem:Sigma-th-simple}, if the representation $A_\mu$ is not simple, then it must either contain a subrepresentation of the form $L(\alpha_i)$, or it must have a quotient of the form $L(\alpha_i)$. Let us show that both situations are impossible when $\mu$ is a $J$-core.

Assume that $A_\mu$ has a subrepresentation isomorphic to $L(\alpha_i)$. Let $v$ be a vector that spans this subrepresentation. We can write $v=\sum\limits_{j\in \ZM, j\equiv i \mod\ell}v_j$, where $v_j\in V_j$. Take $j$ in this decomposition such that $v_j\ne 0$. Then the vector $v_j$ also spans a subrepresentation of $A_\mu$ isomorphic to $L(\alpha_i)$.

Let $t$ be the number of boxes of $\mu$ with the $\infty$-residue $j$. Write $v_j=\sum_{r=1}^t\lambda_r v_{r,j}$. Then $X(v)=0$ is only possible when $\lambda_1=\ldots=\lambda_{t-1}=0$, so the vector $v_{t,j}$ spans $L(\alpha_i)$.

Assume $j<0$.  Since the box $b$ corresponding to the vector $v_{t,j}$ cannot be removable, the diagram of $\mu$ either contains the box below $b$ or the box on the right of $b$. In the first case we must have $X(v_{t,j})\ne 0$ and in the second case we must have $Y(v_{t,j})\ne 0$. This is a contradiction.

Assume $j>0$. Then $X(v_{t,j})\ne 0$. This is a contradiction.

Assume $j=0$. Then, since $\beta_1\ne 0$, $Y(v_{t,0})\ne 0$ is only possible for $t=1$. However, this implies that $\mu$ contains only one hook (i.e., we have $k=1$). Since the box $b$ corresponding to the vector $v_{1,0}$ cannot be removable, the diagram of $\mu$ either contains the box below $b$ or the box on the right of $b$. The first case is not possible because it implies $X(v_{1,0})\ne 0$. In the second case we must have $\theta_1+\theta_2+\ldots+\theta_{a_1}=0$. However, this implies $\theta_{a_1}=0$ and then the unique box with $\infty$-residue $a_1$ is removable. This is a contradiction.

Now, assume that $A_\mu$ has a quotient isomorphic to $L(\alpha_i)$. Then the dual representation $A_\mu^*$ contains a submodule isomorphic to $L(\alpha_i)$. An argument as above show that this is impossible if $A_\mu$ is a $J$-core.

\end{proof}

Denote by $A'_\mu$ the semisimplification of $A_\mu$, i.e., $A'_\mu$ is the direct sum of the Jordan-H\"older subquotients of $A_\mu$.

\begin{coro} 
\label{coro:structure-Amu}
	Assume $\mu\in \PC$ and set $\lambda=\Core_J(\mu)$. Then the representation $A'_\mu$ has the following decomposition in a direct sum of simple representaions
	$$
	A'_\mu=A_\lambda\oplus \bigoplus_j L(\alpha_j),
	$$
where the sum is taken by the multiset of $\ell$-residues of $\mu\smallsetminus\lambda$. 
\end{coro}

\begin{defi}
	We say that the representation $(X,Y,x,y)$ of $\widehat Q_\ell$ is $\ZM$-\emph{gradable} if it is isomorphic to the image by $\iota$ (see Definition \ref{def:map-i-inf}) of some representaion $L$ of $\widehat Q_\infty$. In this case we say that $L$ is a \emph{graded lift} of $(X,Y,x,y)$.
\end{defi}

A  $\ZM$-gradable representation yields a $\CM^\times$-fixed point in $\XC_\th(\db)$.

\begin{lem}
Assume that $(X,Y,x,y)$ is simple and $\ZM$-gradable. Then its $\ZM$-grading is unique. 
\end{lem} 
\begin{proof}
Since we assume $a\ne 0$, the vector $v=x(1)$ must be nonzero (here $1$ is a vector spanning the $\infty$-component of the representration, which is isomorphic to $\CM$). Then $v$ should be in $\ZM$-degree $0$. Since the representation is simple, the vectors of the form $X^{a_1}Y^{b_1}\ldots X^{a_k}Y^{b_k}(v)$ and the vector $1$ span the representation. But then vector $X^{a_1}Y^{b_1}\ldots X^{a_k}Y^{b_k}(v)$ must be in $\ZM$-degree $b_1-a_1+\ldots+b_k-a_k$. This shows that the $\ZM$-grading is unique. 
\end{proof}

\begin{exemple}
\label{ex:Res-Amu}
If $\mu$ is a $J$-core, then the representration $A_\mu$ is simple. It is $\ZM$-gradable by construction. Its graded lift $A^\infty_\mu$ is unique. The $\ZM$-graded dimension of the graded lift $A^\infty_\mu$ is $\Res_\infty(\mu)$.
\end{exemple}

\begin{coro} 
\label{coro:isom-if-same-Jcore}
For $\mu_1,\mu_2\in \PC_\nu[n\ell+|\nu|]$, the repsesentations $A'_{\mu_1}$ and $A'_{\mu_2}$ are isomorphic if and only if $\mu_1$ and $\mu_2$ have the same $J$-cores.
\end{coro}
\begin{proof}
Let $\lambda_1$ and $\lambda_2$ be the $J$-cores of $\mu_1$ and $\mu_2$ respectively.

Assume that $A'_{\mu_1}$ and $A'_{\mu_2}$ are isomorphic. We see from Corollary \ref{coro:structure-Amu} that the representaions $A_{\lambda_1}$ and $A_{\lambda_2}$ are also isomorphic. Now, Example \ref{ex:Res-Amu} implies $\Res_\infty(\lambda_1)=\Res_\infty(\lambda_2)$, this yields $\lambda_1=\lambda_2$. 

Now, assume that we have $\lambda_1=\lambda_2$. Since we have $\mu_1,\mu_2\in \PC_\nu[n\ell+|\nu|]$, the partitions $\mu_1$ and $\mu_2$ have the same residues equal to $\Res_\ell(\nu)+n\delta_\ell$. Then $\mu_1\smallsetminus\lambda_1$ and $\mu_2\smallsetminus\lambda_2$ have the same residues. Then Corollary \ref{coro:structure-Amu} implies that $A'_{\mu_1}$ and $A'_{\mu_2}$ are isomorphic.
	
\end{proof}

\begin{rema}
For each partition $\mu$, we have a $\CM^\times$-fixed point $[A'_\mu]\in \XC_\th(\Res_\ell(\mu))$ presented by the representation $A'_\mu$. 
	
Set $\db=\Res_\ell(\nu)+n\delta_\ell$. Assume $\db\in E_\th$, see Remark \ref{rk:couple-equiv-nd}. By \cite[Prop.~8.3~(i)]{gordon quiver}, the $\CM^\times$-fixed points in $\XC_\th(\db)$ are parameterized by $J$-cores of elements of $\PC_\nu[n\ell+|\nu|]$. On the other hand, we have already constructed the same number of $\CM^\times$-fixed points $[A'_\mu]$ for $\mu\in\PC_\nu[n\ell+|\nu|]$, see Corollary \ref{coro:isom-if-same-Jcore}.

This implies that each $\CM^\times$-fixed point in $\XC_\th(\db)$ is of the form $[A'_\mu]$.
\end{rema}

\subsection{Parameterization of symplectic leaves}

\label{subs:param-sympl}

\begin{lem}
\label{lem:equiv-simple-nd}
The following conditions are equivalent.

$(a)$ The pair $(\db,\th)$ is equivalent to a pair of the form $(n\delta_\ell,\theta')$ with $n\geqslant 0$.

$(b)$ We have $\db\in E_\th$.
\end{lem}
\begin{proof}
$b)$ implies $a)$ by Remark \ref{rk:couple-equiv-nd}.

Now, let us prove that $a)$ implies $b)$. Assume that $(\db,\th)$ satisfies $a)$. Since, the isomorphism \eqref{eq:iso-lusztig} sends simple representations to simple representations by construction, it is enough to assume $\db=n\delta_\ell$. Let $\db_0$ be associated to $\db=n\delta_\ell$ and $\th$ as in Proposition \ref{prop:can-decomp-of-d}. Then $b)$ is equivalent to $\db_0=\db$. 

Assume that we have $\db_0\ne \db$. Since, the couple $(\db_0,\th)$ satisfies $b)$, it also satisfies $a)$. So, it must be equivalent to some couple of the form $(n'\delta_\ell,\th')$. Since we have $\db_0-n'\delta_\ell\in \ZM_{\geqslant 0}^{\ZM/\ell\ZM}$ and $0\ne n\delta_\ell-\db_0\in \ZM_{\geqslant 0}^{\ZM/\ell\ZM}$,  we get $n>n'$.

Now, we get $\XC_{\th}(n\delta_\ell)\simeq \XC_{\th}(\db_0)$ by Proposition \ref{prop:can-decomp-of-d} $c)$ and we have $\XC_{\th}(\db_0)\simeq \XC_{\th'}(n'\delta_\ell)$ by \eqref{eq:iso-lusztig}. This is impossible because by Lemma \ref{lem:dim-CM} we have $\dim \XC_{\th}(n\delta_\ell)=2n$, $\dim \XC_{\th'}(n'\delta_\ell)=2n'$ and $n'<n$.
\end{proof}

\begin{lem}
\label{lem:descr-Eth}
Assume that $\theta$ is $J$-standard. Then we have $\db\in E_\th$ if and only if we have 
$$
\db=\Res_\ell(\nu)+r\delta_\ell
$$
with $r\geqslant 0$ and $\nu\in \CC_\ell\cap\CC_J$.
\end{lem}
\begin{proof}
The parabolic subgroup $W_J$ of $W^{\rm aff}_\ell$ is the stabilizer of $\theta$ in $W^{\rm aff}_\ell$.
Write $\db=\Res_\ell(\nu)+r\delta_\ell$ as in \eqref{eq:d-via-core-n}, we have $r\in\ZM$ and $\nu\in\CC_\ell$.

\medskip
Assume $\db\in E_\th$. Then Lemma \ref{lem:equiv-simple-nd} implies that $r\geqslant 0$ and that we can find $x\in W_\ell^\aff$ (see  Remark \ref{rem:equiv=shortest}) such that $x(\db)=r\delta_\ell$ and such that $x$ is the shortest element in the coset $xW_J\in W^{\rm aff}_\ell/W_J$.  

Let $w$ be the shortest element in $x^{-1}W_\ell$. We have $\nu=x^{-1}(\emptyset)=w(\emptyset)$. Assume that $\nu$ is not a $J$-core. Then we have $|s_i(\nu)|<|\nu|$ for some $i\in J$, this corresponds to the case $(3)$ in  Remark \ref{rem:W-act-lcores}. Then Lemma \ref{lem:nu-w-i} implies $l(s_iw)<l(w)$. Then we also have $l(s_ix^{-1})<l(x^{-1})$ or equivalently $l(xs_i)<l(x)$. This contradicts to the fact that $x$ is the shortes element in $xW_J$. Then $\nu$ must be a $J$-core.

\medskip
Now, assume that we have $\db=\Res_\ell(\nu)+r\delta_\ell$ for $r\geqslant 0$ and $\nu\in\CC_\ell\cap\CC_J$. Let $w$ be the element of $W^\aff_\ell$ such that $w(\emptyset)=\nu$ and such that $w$ is the shortest element in $wW_\ell$. It is inough to prove that $w$ in the shortest element in $W_Jw$. Indeed, if we prove this, then by Remark \ref{rem:equiv=shortest} we have $(\db,\th)\sim (w^{-1}(\db),w^{-1}(\th))=(r\delta_\ell,w^{-1}(\th))$ and then by Lemma \ref{lem:equiv-simple-nd} we have $\db\in E_\th$.

 Since $\nu$ is a $J$-core, for each $i\in J$ we have $|s_i(\nu)|\geqslant |\nu|$. This means that for each $i\in J$, we are either in the situation $(1)$ or in the situation $(2)$ of Remark \ref{rem:W-act-lcores}. In both cases Lemma \ref{lem:nu-w-i} yields $l(s_iw)>l(w)$. 

\end{proof}

\begin{rema}
\label{rem:dim-nd+nu}
Assume that $\th$ is $J$-standard and fix $\db\in E_\th$. By the lemma above, we can write $\db$ in the form $\db=\Res_\ell(\nu)+n\delta_\ell$ with $n\geqslant 0$ and $\nu\in \CC_\ell\cap\CC_J$. Then by Lemma \ref{lem:equiv-simple-nd}, the couple $(\db,\th)$ is equivalent to $(n\delta_\ell,\th')$ for some $\th'\in \CM^{\ZM/\ell\ZM}$. 
Then Lemma \ref{lem:dim-CM} implies that the variety $\XC_\th(\Res_\ell(\nu)+n\delta_\ell)$ is normal of dimension $2n$.
\end{rema}

We see that the elements of $E_\th$ are in bijection with the couples $(\nu,r)$ where $\nu$ is an $\ell$-core that is a $J$-core and $r\in \ZM_{\geqslant 0}$.

Assume that $\th$ is $J$-standard. Then we have a partial order $\succcurlyeq$ on $E_\th$ given by $\db\succcurlyeq\db'$ if $\db-\db'\in \sum_{j\in J} \ZM_{\geqslant 0} \alpha_j$. In other words, we have $\db\succcurlyeq\db'$ if and only if $\LG^\db_{\db'}\ne \emptyset$. Using the bijection above, we may consider the order $\succcurlyeq$ as an order on the set $(\CC_\ell\cap\CC_J)\times \ZM_{\geqslant 0}$. 

\begin{lem}
We have $(\nu_1,r_1)\succcurlyeq (\nu_2,r_2)$ if and only if we have $r_1\geqslant r_2$ and there exists a partition $\lambda\in \PC_{\nu_1}[|\nu_1|+\ell(r_1-r_2)]$ such that $\Core_J(\lambda)=\nu_2$.
\end{lem}
\begin{proof}
Assume $(\nu_1,r_1)\succcurlyeq (\nu_2,r_2)$. Then we have $\dim \XC_\th(\Res_\ell(\nu_1)+r_1\delta_\ell)=2r_1$ and  $\dim \XC_\th(\Res_\ell(\nu_2)+r_2\delta_\ell)=2r_2$ by Remark \ref{rem:dim-nd+nu}. By Corollary \ref{coro:norm-cl-leaf} and its proof, the normalization of the closure of the symplectic leaf $\LG^{\Res_\ell(\nu_1)+r_1\delta_\ell}_{\Res_\ell(\nu_2)+r_2\delta_\ell}$ is isomorphic to $\XC_\th(\Res_\ell(\nu_2)+r_2\delta_\ell)$. In particular, 
$$
\dim\XC_\th(\Res_\ell(\nu_1)+r_1\delta_\ell)\geqslant\dim\LG^{\Res_\ell(\nu_1)+r_1\delta_\ell}_{\Res_\ell(\nu_2)+r_2\delta_\ell}
$$ 
implies $r_1\geqslant r_2$.

Now, $(\nu_1,r_1)\succcurlyeq (\nu_2,r_2)$ implies $\Res_\ell(\nu_1)+r_1\delta_\ell \succcurlyeq \Res_\ell(\nu_2)+r_2\delta_\ell$ and then $(\nu_1,r_1-r_2)\succcurlyeq (\nu_2,0)$. This means that the variety $\XC_\th(\Res_\ell(\nu_1)+(r_1-r_2)\delta_\ell)$ has a symplectic leaf $\LG^{\Res_\ell(\nu_1)+(r_1-r_2)\delta_\ell}_{\Res_\ell(\nu_2)}$. This simplectic leaf is $0$-dimensional, so it is a $\CM^\times$-fixed point. Then by \S \ref{subs:C-fixed-points}, this should be a point of the form $[A'_\lambda]$ for some $\lambda\in \PC_{\nu_1}[|\nu_1|+\ell(r_1-r_2)]$. By Corollary \ref{coro:structure-Amu} we have $\dim^\reg(A'_\lambda)=\Res_\ell(\Core_J(\lambda))$. Then $[A'_\lambda]\in \LG^{\Res_\ell(\nu_1)+(r_1-r_2)\delta_\ell}_{\Res_\ell(\nu_2)}$ implies $\Core_J(\lambda)=\nu_2$.

\medskip
Inversly, if $r_1\geqslant r_2$ and if there exists such a partition $\lambda$, then the $\CM^\times$-fixed point $[A'_\lambda]$ of $\XC_\th(\Res_\ell(\nu_1)+(r_1-r_2)\delta_\ell)$ is a simplectic leaf. Since $\dim^\reg(A'_\lambda)=\Res_\ell(\Core_J(\lambda))=\Res_\ell(\nu_2)$, this is the symplectic leaf $\LG^{\Res_\ell(\nu_1)+(r_1-r_2)\delta_\ell}_{\Res_\ell(\nu_2)}$. Then we have $(\nu_1,r_1-r_2)\succcurlyeq (\nu_2,0)$. This implies $\Res_\ell(\nu_1)+(r_1-r_2)\delta_\ell \succcurlyeq \Res_\ell(\nu_2)$ and then $(\nu_1,r_1)\succcurlyeq (\nu_2,r_2)$.
\end{proof}

Assume that $\theta$ is $J$-standard and $\db\in E_\theta$. Write $\db=\Res_\ell(\nu)+n\delta_\ell$, $\nu\in\CC_\ell\cap\CC_J,n\geqslant 0$.
\begin{coro}
For $\db'\in \ZM/\ell\ZM$, the following conditions are equivalent.

$(a)$ We have $\LG^\db_{\db'}\ne \emptyset$.

$(b)$ There exists a partition $\lambda\in\PC_\nu[n'\ell+|\nu|]$ for some $n'\in[0;n]$ such that we have $\db'=\Res_\ell(\Core_J(\lambda))+(n-n')\delta_\ell$.
\end{coro}
\begin{proof}
Write $\db'=\Res_\ell(\nu')+r'\delta_\ell$. Then $\LG^\db_{\db'}\ne \emptyset$ is equivalent to $(\nu,n)\succcurlyeq (\nu',r')$. By the lemma above, this is equivalent to $n\geqslant r'$ and the existence of a partition $\lambda\in \PC_{\nu}[\ell(n-r')+|\nu|]$ such that $\Core_J(\lambda)=\nu'$. Moreover, the condition $\Core_J(\lambda)=\nu'$ is equivalent to $\Res_\ell(\Core_J(\lambda))=\Res_\ell(\nu')=\db'-r'\delta_\ell$. Now we see that $(a)$ is equivalent to $(b)$ with $n'=n-r'$.


\end{proof}

In particular, we see that the simplectic leaves of $\XC_\th(\db)$ are parametrized by $\ell$-cores of $J$-cores of elements of $\PC_\nu[n'\ell+|\nu|]$ for $n'\in [0;n]$. Note that by Lemma \ref{lem:lcore-of-Jcore}, that $\ell$-cores of $J$-cores are also $J$-cores. 

In other words, the symplectic leaves of $\XC_\th(\db)$ are paremeterized by a subset of the set $\CC_\ell\cap\CC_J$. This subset is the image of the set $\coprod_{n'=0}^n \PC_\nu[n'\ell+|\nu|]$ by the map $\Core_\ell\circ \Core_J$.

Since each couple $(n\delta_\ell,\th)$ is equivalent to some couple of the form $(\db,\th')$ such that $\th'$ is $J$-standard for some $J$ and $\db\in E_{\th'}$ (see Lemma \ref{lem:equiv-simple-nd}), the description above gives a parameterization of the symplectic leaves of an arbitrary Calogero-Moser space of type $G(\ell,1,n)$ with $a\ne 0$.

\begin{exemple}

Assume $\ell=2$. In this case the set $\CC_2$ of $2$-cores is labelled by nonnegative integers. We have $\CC_2=\{\nu_m,m\in \ZM_{\geqslant 0}\}$ where $\nu_m$ is the partition $\nu_m=(m,m-1,m-2,\ldots,2,1)$ of $m(m+1)/2$. The two possible non-trivial examples of $J$ are $J_0=\{0\}$ and $J_1=\{1\}$. Then the $2$-cores $\nu_2, \nu_4, \nu_6,\ldots$ are $J_0$-cores and not $J_1$-cores, the $2$-cores $\nu_1, \nu_3, \nu_5,\ldots$ are $J_1$-cores and not $J_0$-cores, the $2$-core $\nu_0=\emptyset$ is a $J_0$-core and a $J_1$-core.

Assume that $\th$ is $J$-standard and $\db\in E_\th$. Assume $J=J_1$ and write $\db=\Res_2(\nu_m)+n\delta_2$. Since $
\nu_m$ must be a $J_1$-core, the number $m$ must be odd or zero. Assume that $m$ is odd. 

Let us see which subset of $\CC_2\cap\CC_J$ parameterizes the symplectic leaves of $\XC_\th(\db)$ in this case. 
If $n\leqslant m+1$, then the only possible $\nu'$ that we may get is $\nu'=\nu_m$. This is the case where the variety $\XC_\th(\db)$ is smooth. If $n\geqslant m+2$ then it is also possible to get $\nu'=\nu_{m+2}$. If $n\geqslant 2(m+3)$ then it is also possible to get $\nu'=\nu_{m+4}$, etc. If $n\geqslant k(m+1+k)$ then it is also possible to get $\nu'=\nu_{m+2k}$. Finally, we see that the symplectic leaves of $\XC_\th(\db)$ are labelled by the following subset of $\CC_2\cap\CC_J$: $\{\nu_m,\nu_{m+2},\nu_{m+4},\ldots,\nu_{m+2k}\}$ where $k$ is the maximal nonnegative integer such that $n\geqslant k(m+1+k)$. 

\end{exemple}

\section*{Acknowledgements}

I would like to thank C\'edric Bonnaf\'e for many helpful discussions during my work on this paper. I would also like to thank Dario Mathi\"a for his comments.


\begin{thebibliography}{20}

\bibitem{bellamy2} {\sc G. Bellamy \& T. Schedler}, {\it Symplectic resolutions of quiver varieties}, Selecta Mathematica {\bf 27}, Article number: 36, 2021. 


\bibitem{BJV} {\sc C. Berg, B. Jones \& M. Vazirani},  
{\it A bijection on core partitions and a parabolic quotient of the affine symmetric group}, 
Journal of Combinatorial Theory, Series A, {\bf 116(8)}, 1344-1360, 2009.

\bibitem{Bon} {\sc C. Bonnaf\'e}, 
{\it Automorphisms and symplectic leaves of Calogero-Moser spaces},
preprint arXiv:2112.12405, 2021.

\bibitem{BM} {\sc C. Bonnaf\'e \& R. Maksimau}, 
{\it Fixed points in smooth Calogero-Moser spaces}, Annales de l'Institut Fourier {\bf 71(2)}, 643-678, 2021.




\bibitem{CB} {\sc W. Crawley-Boevey}, 
{\it Geometry of the moment map for representations of quivers}, 
Compo. Math. {\bf 126(3)}, 257-293, 2001.

\bibitem{CB2} {\sc W. Crawley-Boevey}, 
{\it Decomposition of Marsden-Weinstein reductions for representations of quivers}, 
Compo. Math. {\bf 130(2)}, 225-239, 2002.

\bibitem{EG} {\sc P. Etingof \& V. Ginzburg}. 
\emph{Symplectic reflection algebras, Calogero-Moser space, and deformed Harish-Chandra homomorphism}, 
Inventiones Mathematicae \textbf{147(2)}, 243-348, 2002.



\bibitem{gordon quiver} {\sc I. Gordon}, 
{\it Quiver varieties, category $\OC$ for rational Cherednik algebras, and Hecke algebras}, 
Int. Math. Res. Papers, 69 pages, 2008.


\bibitem{Kac} {\sc V. G. Kac},  {\it Infinite Dimensional Lie Algebras - 
An Introduction}, Birkh\"auser, 1983.




\bibitem{lusztig} {\sc G. Lusztig}, 
{\it Quiver varieties and Weyl group actions}, 
Ann. Inst. Fourier {\bf 50}, 461-489, 2000.


\bibitem{martino} {\sc M. Martino}, {\it Symplectic reflection algebras and Poisson geometry}, Doctoral dissertation, ProQuest Dissertations \& Theses, 2006.

\bibitem{Przez} {\sc T. Przezdziecki},  {\it The combinatorics of 
 $\mathbb {C}^* $-fixed points in generalized Calogero-Moser spaces and Hilbert schemes}, Journal of Algebra, {\bf 556}, 936-992, 2020.


\end{thebibliography}
\end{document}